\documentclass[11pt,reqno]{amsart}
\usepackage[a4paper]{anysize}\marginsize{2.5cm}{2.5cm}{2cm}{2cm}
\pdfpagewidth=\paperwidth \pdfpageheight=\paperheight
\allowdisplaybreaks
\usepackage[english]{babel}
\usepackage[utf8]{inputenc}
\usepackage[T1]{fontenc}
\usepackage{amssymb,amsmath,amstext,amsfonts,amsthm}
\usepackage{mathtools}
\usepackage{verbatim}
\usepackage{fancyvrb}
\usepackage[colorlinks=true, urlcolor=blue, linkcolor=blue, citecolor=blue]{hyperref}
\usepackage{hyperref}
\usepackage{nicematrix}
\usepackage{array}
\usepackage{arydshln}
\usepackage{tikz}
\usetikzlibrary{arrows.meta}
\usepackage{alphabeta}

\definecolor{OliveGreen}{rgb}{0,0.6,0}

\numberwithin{equation}{section}
\theoremstyle{plain}
\newtheorem*{theorem*}{Theorem}
\newtheorem{theorem}{Theorem}
\numberwithin{theorem}{section}
\newtheorem{proposition}[theorem]{Proposition}

\newtheorem{remark}[theorem]{Remark}

\makeatletter
\renewcommand\verbatim@font{\normalfont\fontencoding{T1}\ttfamily}
\makeatother

\theoremstyle{definition}
\newtheorem{definition}[theorem]{Definition}
\newtheorem{example}[theorem]{Example}
\newtheorem{notation}[theorem]{Notation}

\newcommand{\E}{\mathbb{E}}
\newcommand{\K}{\mathbb{K}}
\newcommand{\Q}{\mathbb{Q}}

\newcommand{\F}{\mathbb{F}}
\newcommand{\R}{\mathbb{R}}

\newcommand{\sC}{\mathcal{C}}
\newcommand{\sD}{\mathcal{D}}
\newcommand{\sE}{\mathcal{E}}

\newcommand{\sT}{\mathcal{T}}
\newcommand{\sM}{\mathcal{M}}

\makeatletter
\@namedef{subjclassname@2020}{%
	\textup{2020} Mathematics Subject Classification}
\makeatother

\author{Muhammad Ardiyansyah}
\address{Department of Mathematics and Systems Analysis, Aalto University, Espoo, Finland}
\email{muhammad.ardiyansyah@aalto.fi}
\author{Luca Sodomaco}
\address{Department of Mathematics, KTH Royal Institute of Technology, SE-100 44 Stockholm, Sweden}
\email{sodomaco@kth.se}

\subjclass[2020]{62R01, 62H25, 62H22}
\keywords{factor analysis model, higher-order cumulants, symmetric tensors}

\title[Dimensions of Higher Order Factor Analysis Models]{Dimensions of Higher Order Factor Analysis Models}

\date{}

\begin{document}

\begin{abstract}
The factor analysis model is a statistical model where a certain number of hidden random variables, called factors, affect linearly the behaviour of another set of observed random variables, with additional random noise.
The main assumption of the model is that the factors and the noise are Gaussian random variables. This implies that the feasible set lies in the cone of positive semidefinite matrices. In this paper, we do not assume that the factors and the noise are Gaussian, hence the higher order moment and cumulant tensors of the observed variables are generally nonzero. This motivates the notion of $k$th-order factor analysis model, that is the family of all random vectors in a factor analysis model where the factors and the noise have finite and possibly nonzero moment and cumulant tensors up to order $k$. This subset may be described as the image of a polynomial map onto a Cartesian product of symmetric tensor spaces. Our goal is to compute its dimension and we provide conditions under which the image has positive codimension.
\end{abstract}

\maketitle

\section{Introduction}
In statistics, factor analysis is a method for describing statistical models where the involved random variables can be arranged into two distinct groups. On one hand, one considers a vector $X\in\R^p$ of {\em observed} random variables $X_i$. On the other hand, the variables $X_i$ are conditionally independent given another vector $Y\in\R^m$ of {\em hidden} or {\em latent} random variables $Y_j$, usually called {\em factors}. In applications, it is common that the number of factors $m$ is considerably smaller than the number of observed variables $n$. Factor analysis may be addressed as a dimension reduction technique where the number of dimensions is specified by the user, see \cite{spearmann1904,spearman1961abilities}.
In those two papers, Charles Spearman introduced the concept of the factor analysis model.
Spearman noticed the huge variety of measures for cognitive study, including visuo-spatial skills, artistic abilities, and reasoning. Through the factor analysis model, he was curious if the underlying general intelligence variable, which is called the ``g'' factor, and specific abilities variable, which is called the ``s'' factor, could explain them all.
In applications, factor analysis is used in many fields such as behavioral and social sciences \cite{braglio2020factor}, medical sciences \cite{samal1989clinical}, economics \cite{bai2015application}, and geography \cite{clark1974application} as a result of the technological advancements of computers.

In general, statistical models rely on certain sets of assumptions. 
In the case of factor analysis model, each variable $X_i$ is a linear combination of the factors $Y_j$ with some independent noise, namely
\begin{equation}\label{eq: linear dependence X with Y}
X=\Lambda Y+\varepsilon\,
\end{equation}
for some unknown coefficient matrix $\Lambda=(\lambda_{ij})\in \R^{p\times m}$, whose entries are sometimes referred as {\em factor loadings}, and for some ``noise'' random vector $\varepsilon\in\R^p$.
In particular, several observed variables $X_i$ might be measures of the same factor $Y_j$.
The factor analysis model may be regarded as a special instance of a much more general graphical model, where the components of a certain random vector $Z$ interact with each other, and their interaction is encoded by the edges of a directed acyclic graph with vertex set equal to the components of $Z$ \cite{Robeva_2021}. In our setting, the random vector $Z$ is the joint vector $(X,Y)$, and the interactions between $X$ and $Y$ are described by a directed bipartite graph where all edges are directed from elements of $Y$ to elements of $X$, as in Figure \ref{fig: bipartite}.

\begin{figure}
\centering
\begin{tikzpicture}
\begin{scope}[every node/.style={circle,thick,draw}]
    \node (Y1) at (-2,1.2) {$Y_1$};
    \node (Y2) at (-2,0) {$Y_2$};
    \node (Y3) at (-2,-1.2) {$Y_3$};
    \node (X1) at (2,2) {$X_1$};
    \node (X2) at (2,1) {$X_2$};
    \node (X3) at (2,0) {$X_3$} ;
    \node (X4) at (2,-1) {$X_4$} ;
    \node (X5) at (2,-2) {$X_5$} ;
\end{scope}
\begin{scope}[every node/.style={fill=white,circle},
              every edge/.style={draw=red,thick}]
    \path [->] (Y1) edge (X1);
    \path [->] (Y1) edge (X2);
    \path [->] (Y1) edge (X3);
    \path [->] (Y1) edge (X4);
    \path [->] (Y1) edge (X5);
\end{scope}
\begin{scope}[every node/.style={fill=white,circle},
              every edge/.style={draw=blue,thick}]
    \path [->] (Y2) edge (X1);
    \path [->] (Y2) edge (X2);
    \path [->] (Y2) edge (X3);
    \path [->] (Y2) edge (X4);
    \path [->] (Y2) edge (X5);
\end{scope}
\begin{scope}[every node/.style={fill=white,circle},
              every edge/.style={draw=green,thick}]
    \path [->] (Y3) edge (X1);
    \path [->] (Y3) edge (X2);
    \path [->] (Y3) edge (X3);
    \path [->] (Y3) edge (X4);
    \path [->] (Y3) edge (X5);
\end{scope}
\end{tikzpicture}
\caption{Example of a bipartite graph describing the interactions between the observed variables $X_i$ and the factors $Y_i$, for $p=5$ and $m=3$.}\label{fig: bipartite}
\end{figure}
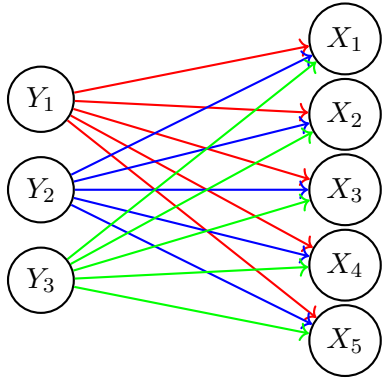

Frequently, the parametric representations of a statistical model are useful for statistical inferences. In factor analysis model, the vector $Y$ and the noise vector $\varepsilon$ are assumed to be Gaussian with mean zero. Therefore, all the information about $X$ is encoded by its covariance matrix, which can be written as the sum of a positive diagonal matrix and a symmetric matrix of a certain rank.
These considerations produce the parametrization of the factor analysis model in \cite[Proposition 1]{drton2007algebraic}.

In this paper we drop the assumption of Gaussianity of the factors and of the noise vector. In particular, in order to understand all the statistical information about the observed variables $X_i$ we need to consider the higher order moment and cumulant tensors of $X$.
We define the {\em $k$th-order cumulant (moment) factor analysis model} to be the set of tuples of length $k-1$ of symmetric tensors of orders $2,3,\ldots,k$ that are the cumulant (moment) tensors of some random vector $X$ of observed variables, see Definition \ref{def: k order cumulant and moment factor analysis model}.
Since we are assuming that the factors $Y_i$ are mutually independent, we have that all higher order cumulant tensors of $Y$ are diagonal tensors. The same property does not hold for moments of order higher than 3. Therefore, similarly as in the more general linear non-Gaussian graphical model studied in \cite[Definition 4]{Robeva_2021}, the higher order cumulants of the vector of observed variables $X$ have a nice polynomial parametrization as the sum of a diagonal tensor and a symmetric tensor obtained as the Tucker product between a diagonal tensor and a rectangular matrix. This parametrization is defined rigorously in \eqref{eq: def phi k}.
Using the described parametrization, we are able to compute the dimension of the $k$th-order cumulant factor analysis model, which for $k=2$ was computed in \cite[Section 2]{drton2007algebraic}.
What is more, using the relations in \eqref{eq: def cumulants} and in \eqref{eq: moments from cumulants}, we can polynomially map the $k$th-order cumulant factor analysis model to the $k$th-order moment factor analysis model, and vice versa.
In particular, the dimensions of the two models coincide. For this reason, we may talk about the {\em dimension of the $k$th-order factor analysis model}.
The dimension clearly depends on the maximal order $k$ of cumulant or moment tensors considered, on the number $p$ of observed variables, and on the number $m$ of factors. In particular, we provide conditions on these parameters under which the $k$th-order factor analysis model has positive codimension, namely when the previous parametrization is ``non-filling''.

One of the most important tasks in statistical modeling is to perform model selection which is the set of rules to select the best model among a set of candidate models. In model selection, we assign a score to each model that depends on its log-likelihood and number of parameters and then we choose the model with the lowest score. The most popular model selection criteria include Akaike Information Criterion (AIC), Hannan-Quinn Information Criterion (HQIC), and Bayesian Information Criterion (BIC) \cite{akaike1998information, burnham2004multimodel,sugiura1978further, schwarz1978estimating}. Therefore, it is crucial to know the dimension of a statistical model. Moreover, knowing the codimension of the model is more desirable since it can be used to measure the complexity of a statistical model.

The paper is organized as follows. After setting up notations, we recall some important properties of moment and cumulant tensors of random vectors, and we define the higher order factor analysis model.
In Section \ref{sec: dimension higher order model}, we compute the dimension of the $k$th-order factor analysis model in Theorem \ref{thm: dimension k order factor analysis model}. Secondly, we provide conditions on the parameters $k$, $p$, and $m$ such that the $k$th-order factor analysis model has positive codimension in the space of tuples of length $k-1$ of symmetric tensors of degrees $2,3,\ldots,k$. Finally, we provide a Macaulay2 code \cite{GS} to verify numerically the previous formulas.

\section{Preliminaries on higher order factor analysis models}\label{sec: prelim}

First, we set up the main notations used throughout the paper.

\begin{notation}
For any integer $n\ge 1$, we denote by $[n]$ the set $\{1,\ldots,n\}$. The ground field used throughout the paper is mainly the field of real numbers $\R$, although the preliminary definitions may be stated for an arbitrary field $\K$. By a {\em tensor of format $m_1\times\cdots\times m_r$ with entries in $\K$} we mean an $r$-dimensional array $\sT=(t_{i_1\cdots i_r})$ filled with entries in $\K$. The vector space of such tensors is usually denoted by $\K^{m_1\times\cdots\times m_r}$ and is isomorphic to the tensor product $\bigotimes_{j=1}^r\K^{m_j}$. Assuming $m_1=\cdots=m_r=m$, we say that a tensor $\sT\in(\K^m)^{\otimes r}$ is {\em symmetric} if its entries $t_{i_1\cdots i_r}$ do not change after a permutation of the indices $i_1,\ldots,i_r$. Symmetric tensors in $(\K^m)^{\otimes r}$ form a subspace denoted by $\mathrm{Sym}^r(\K^m)$. Furthermore, we denote by $\Delta^r(\K^m)$ the subset of $\mathrm{Sym}^r(\K^m)$ of diagonal tensors. In particular, the subset $\Delta^r(\K^m)$ can be identified with $\K^m$.
\end{notation}

\begin{definition}
The $r$th-order {\em moment} tensor $\sM_X^{(r)}\in(\R^p)^{\otimes r}$ of a random vector $X\in\R^p$ is defined by the element-wise equation
\begin{equation}
(\sM_X^{(r)})_{j_1\cdots j_r} = \mathrm{mom}(X_{j_1},\ldots,X_{j_r}) = \E[X_{j_1}\cdots X_{j_r}]\,.
\end{equation}
\end{definition}

\begin{definition}
The $r$th-order {\em cumulant} tensor $\sC_X^{(r)}\in(\R^p)^{\otimes r}$ of a random vector $X\in\R^p$ is defined by the element-wise equation
\begin{equation}\label{eq: def cumulants}
(\sC_X^{(r)})_{j_1\cdots j_r} = \mathrm{cum}(X_{j_1},\ldots,X_{j_r}) = \sum_{(A_1,\ldots,A_L)}(-1)^{L-1}(L-1)!\prod_{i=1}^L \mathrm{mom}((X_j)_{j\in A_i})\,,
\end{equation}
where the sum in \eqref{eq: def cumulants} is taken with respect to all partitions $\{A_1,\ldots,A_L\}$ of the set $\{j_1,\ldots,j_r\}$.
\end{definition}

\begin{remark}\label{rmk: polynomial bijection}
Suppose that all moment tensors $\sM_X^{(r)}$ of $X$ are known up to order $k$. Then equation \eqref{eq: def cumulants} allows us to compute all $r$th-order cumulant tensors $\sC_X^{(r)}$ of $X$, up to order $k$. What is more, equation \eqref{eq: def cumulants} defines a polynomial map from the set of moments up to order $k$ to the set of cumulants up to order $k$ of $X$. This polynomial map has a polynomial inverse, which is defined by the element-wise equation (see \cite[Section 2.3.4]{mccullagh1987tensor} for further details)
\begin{equation}\label{eq: moments from cumulants}
(\sM_X^{(r)})_{j_1\cdots j_r} = \mathrm{mom}(X_{j_1},\ldots,X_{j_r}) = \sum_{(A_1,\ldots,A_L)}\prod_{i=1}^L\mathrm{cum}((X_j)_{j\in A_i})\,.
\end{equation}
\end{remark}

Given a real random vector $X\in\R^p$, the first three cumulant tensors of $X$ are defined by the identities
\begin{align}\label{eq: first moments and cumulants of X}
\begin{split}
(\sC_X^{(1)})_{j} &= \E[X_j]\,,\\
(\sC_X^{(2)})_{j_1j_2} &= \E[X_{j_1}X_{j_2}] - \E[X_{j_1}]\E[X_{j_2}]\,,\\
(\sC_X^{(3)})_{j_1j_2j_3} &= \E[X_{j_1}X_{j_2}X_{j_3}] - \E[X_{j_1}]\E[X_{j_2}X_{j_3}]-\E[X_{j_2}]\E[X_{j_1}X_{j_3}]\\
&\quad-\E[X_{j_3}]\E[X_{j_1}X_{j_2}]+2\E[X_{j_1}]\E[X_{j_2}]\E[X_{j_3}]\,.\\
\end{split}
\end{align}
If additionally we impose that $X$ has mean zero, namely $\E[X_i]=0$ for all $i\in[p]$, we see immediately from the previous identities that $\sC_X^{(r)}=\sM_X^{(r)}$ for $r\in[3]$. Instead, the two tensors are in general different for $r\ge 4$, as we can see from the expression of the fourth cumulant tensor when $X$ has mean zero:
\begin{align}
\begin{split}
(\sC_X^{(4)})_{j_1j_2j_3j_4} &= \E[X_{j_1}X_{j_2}X_{j_3}X_{j_4}] - \E[X_{j_1}X_{j_2}]\E[X_{j_3}X_{j_4}]\\
&\quad-\E[X_{j_1}X_{j_3}]\E[X_{j_2}X_{j_4}]-\E[X_{j_1}X_{j_4}]\E[X_{j_2}X_{j_3}]\,.
\end{split}
\end{align}
It is almost immediate to check from their definitions that both moment tensors and cumulant tensors are symmetric tensors. In order to state another important property shared by moment and cumulant tensors, we need to recall a natural operation between tensors and matrices that generalizes the classical operation of matrix multiplication.

\begin{definition}\label{def: Tucker product}
Let $\sT$ be a tensor of format $m_1\times\cdots\times m_r$ with entries in a field $\K$. For every $\ell\in[r]$ consider a matrix $U_{\ell}=(u_{ij}^{(\ell)})\in\K^{m_\ell\times p_\ell}$. The {\em Tucker product} (or {\em multilinear multiplication}) of $\sT$ by $(U_1,\ldots,U_r)$ is the tensor $\sT\bullet(U_1,\ldots,U_r)$ of format $p_1\times\cdots\times p_r$ whose entry $(i_1,\ldots,i_r)$ is
\[
(\sT\bullet(U_1,\ldots,U_r))_{i_1\cdots i_r} = \sum_{j_1=1}^{m_1}\cdots\sum_{j_r=1}^{m_r}t_{j_1\cdots j_r}u_{j_1i_1}^{(1)}\cdots u_{j_ri_r}^{(r)}\,.
\]
In particular for $r=2$ we have that $\sT\bullet(U_1,U_2)=U_1^T\,\sT\,U_2$. If $m_1=\cdots=m_r=m$ and $p_1=\cdots=p_r=p$, we use the shorthand $\sT\bullet^r U$ to denote the Tucker product $\sT\bullet(U,\ldots,U)$, where $U\in\K^{m\times p}$ is repeated $r$ times.
\end{definition}

\begin{example}\label{ex: Tucker diagonal tensor}
In our paper, we will have $m_1=\cdots=m_r=m$, $p_1=\cdots=p_r=p$, and we will compute the Tucker product of a diagonal tensor $\sD\in\Delta^r(\R^m)$ by the transpose of the matrix $\Lambda=(\lambda_{ij})\in\R^{p\times m}$ introduced in \eqref{eq: linear dependence X with Y}. If we denote by $\delta_1,\ldots,\delta_m$ the diagonal entries of $\sD$, then
\begin{equation}
(\sD\bullet^r\Lambda^T)_{i_1\cdots i_r} = \sum_{\ell=1}^m\delta_\ell\,\lambda_{i_1,\ell}\cdots\lambda_{i_r,\ell}\,.
\end{equation}
The tensor $\sD\bullet^r\Lambda^T$ is symmetric, so it suffices to study the entries with multi-index $(i_1,\ldots,i_r)$ such that $i_1\le\cdots\le i_r$. Without loss of generality, we will assume that $\Lambda$ is a lower-triangular matrix, that is $\lambda_{ij}=0$ if $j>i$. Then
\begin{equation}\label{eq: Tucker product used in the main proof}
(\sD\bullet^r\Lambda^T)_{i_1\cdots i_r} = \sum_{\ell=1}^{\min\{m,i_1\}}\delta_\ell\,\lambda_{i_1,\ell}\cdots\lambda_{i_r,\ell}\quad\forall\,i_1\le\cdots\le i_r\,.
\end{equation}
\end{example}

In the following proposition we recall some important properties of cumulant tensors, which are discussed more in detail in \cite[Chapter 5]{comon2010handbook}.

\begin{proposition}\label{prop: properties cumulant tensors}
Let $X,Z\in\R^p$, $Y\in\R^m$ be random vectors.
\begin{itemize}
    \item[$(a)$] If $X=\Lambda Y$ for some matrix $\Lambda\in\R^{p\times m}$, then
    \begin{equation}
    \sC_X^{(r)} = \sC_Y^{(r)}\bullet^r\Lambda^T\quad\forall\,r\ge 1\,.
    \end{equation}
    \item[$(b)$] If $X,Z\in\R^p$ are mutually independent, then
    \begin{equation}
    \sC_{X+Z}^{(r)} = \sC_{X}^{(r)}+\sC_{Z}^{(r)}\quad\forall\,r\ge 1\,.
    \end{equation}
    \item[$(c)$] If the components $X_i$ of $X$ are mutually independent, then $\sC_{X}^{(r)}$ is a diagonal tensor for all $r\ge 1$, namely $(\sC_X^{(r)})_{j_1\cdots j_r}\neq 0$ only if $j_1=\cdots=j_r$.
\end{itemize}
\end{proposition}

It is worth mentioning that property $(a)$ is valid also when $\sC_X^{(r)}$ and $\sC_Y^{(r)}$ are replaced by the moment tensors $\sM_X^{(r)}$ and $\sM_Y^{(r)}$. Instead, properties $(b)$ and $(c)$ are not valid in general for moment tensors of order $r\ge 3$.

The statistical model we consider in this paper is the factor analysis model described via equation \eqref{eq: linear dependence X with Y} in the introduction, where all correlations among the coordinates of a random vector $X\in\R^p$ are due to another random vector $Y\in\R^m$ whose components are usually called {\em factors}. Correlations between $X$ and $Y$ are encoded by the coefficient matrix $\Lambda=(\lambda_{ij})\in \R^{p\times m}$, and by the random vector $\varepsilon\in\R^p$.

A fundamental assumption on the (classical) factor analysis model is that $Y$ and $\varepsilon$ are random {\em Gaussian} vectors. In this paper we want to drop this assumption, hence we allow distributions with nonzero moments or cumulants up to order $k\ge 2$. For this reason we consider the following modified model.

\begin{definition}\label{def: our model}
Let $k\ge 2$ be an integer. A {\em $k$th-order factor analysis model} is a family of random vectors $X$ of observed variables that are correlated to another vector $Y$ of hidden variables (called factors) via equation \eqref{eq: linear dependence X with Y}, where $\varepsilon$ is a noise component. The model relies on the following assumptions:
\begin{enumerate}
    \item All moment and cumulant tensors of $Y$ and $\varepsilon$ exist and are finite up to order $k$.
    \item The vectors $Y$ and $\varepsilon$ are independent each other.
    \item The components of $Y$ are mutually independent, and similarly for $\varepsilon$.
    \item The vectors $Y$ and $\varepsilon$ have mean equal to zero.
\end{enumerate}
\end{definition}

The last assumption in Definition \ref{def: our model} and the relations in \eqref{eq: first moments and cumulants of X} imply that $\sC_{Y}^{(r)}=\sM_{Y}^{(r)}$ and $\sC_{\varepsilon}^{(r)}=\sM_{\varepsilon}^{(r)}$ for all $r\in[3]$, therefore $\sC_{X}^{(r)}=\sM_{X}^{(r)}$ for all $r\in[3]$.
This is no longer true for $k\ge 4$.

Definition \ref{def: our model} has a natural counterpart in terms of cumulant and moment tensors.

\begin{definition}\label{def: k order cumulant and moment factor analysis model}
Let $p$, $m$ and $k$ be nonnegative integers with $k\ge 2$.
\begin{itemize}
    \item[$(a)$] The {\em $k$th-order cumulant factor analysis model} is the subset of tuples $(\sC^{(2)},\ldots,\sC^{(k)})$ of symmetric tensors $\sC^{(r)}\in\mathrm{Sym}^r(\R^p)$ that are the cumulant tensors for some random vector $X\in\R^p$ in the $k$th-order factor analysis model. We denote this subset by $\sC^{(\le k)}_{p,m}$.
    \item[$(b)$] The {\em $k$th-order moment factor analysis model} is the subset of tuples $(\sM^{(2)},\ldots,\sM^{(k)})$ of symmetric tensors $\sM^{(r)}\in\mathrm{Sym}^r(\R^p)$ that are the moment tensors for some random vector $X\in\R^p$ in the $k$th-order factor analysis model. We denote this subset by $\sM^{(\le k)}_{p,m}$.
\end{itemize}
\end{definition}

\begin{proposition}\label{pro: parametrize k order cumulant factor analysis model}
Let $p$, $m$ and $k$ be nonnegative integers with $k\ge 2$. Consider an element $(\sC^{(2)},\ldots,\sC^{(k)})\in\sC^{(\le k)}_{p,m}$. Then for all $r\in\{2,\ldots,k\}$
\begin{equation}
\sC^{(r)} = \sD^{(r)}\bullet^r\Lambda^T + \sE^{(r)}
\end{equation}
for some $\Lambda\in\R^{p\times m}$, $\sD^{(r)}\in\Delta^r(\R^m)$, and $\sE^{(r)}\in\Delta^r(\R^p)$.
Furthermore, for $r=2$ the diagonal matrices $\sD^{(2)}$ and $\sE^{(2)}$ are positive semidefinite.
\end{proposition}
\begin{proof}
Suppose that, for some $r\in\{2,\ldots,k\}$, the tensor $\sC^{(r)}$ is the $r$th-order cumulant tensor $\sC_X^{(r)}$ of some random vector $X\in\R^p$ in the $k$th-order factor analysis model.
The second assumption in Definition \ref{def: our model}, together with Propositions \ref{prop: properties cumulant tensors}$(a)$ and \ref{prop: properties cumulant tensors}$(b)$, say that the linear relation \eqref{eq: linear dependence X with Y} implies the following relation between $\sC_X^{(r)}$ and $\sC_Y^{(r)}$:
\begin{equation}
    \sC_{X}^{(r)} = \sC_{Y}^{(r)}\bullet^r\Lambda^T+\sC_{\varepsilon}^{(r)}\,.
\end{equation}
Furthermore, the third assumption in Definition \ref{def: our model} and Proposition \ref{prop: properties cumulant tensors}$(c)$ imply that both tensors $\sC_{Y}^{(r)}$ and $\sC_{\varepsilon}^{(r)}$ are diagonal for all $r\ge 1$.
The last property for $r=2$ follows because the covariance matrix of a random vector is always positive semidefinite.
This completes the proof.
\end{proof}

The last proposition tells us that the family $\sC^{(\le k)}_{p,m}$ has a natural description as the image of a certain polynomial map. Using this fact, our goal is to compute the dimension and the codimension of $\sC^{(\le k)}_{p,m}$ in the space $\prod_{r=2}^k\mathrm{Sym}^r(\R^p)$. Furthermore, the fact that all tensors $\sE^{(r)}$ and $\sD^{(r)}$ in Proposition \ref{pro: parametrize k order cumulant factor analysis model} are diagonal simplifies a lot the computation of the dimension of $\sC^{(\le k)}_{p,m}$.

Unfortunately, a similar statement does not hold for the $k$th-order moment factor analysis model $\sM^{(\le k)}_{p,m}$: in fact, in this case the assumptions given in Definition \ref{def: our model} do not imply that the moment tensors of $Y$ and $\varepsilon$ of order $k\ge 4$ are diagonal. Nevertheless, as we discussed in Remark \ref{rmk: polynomial bijection}, the polynomial relations in \eqref{eq: def cumulants} and \eqref{eq: moments from cumulants} allow us to conclude that the dimensions of the models $\sC^{(\le k)}_{p,m}$ and $\sM^{(\le k)}_{p,m}$ coincide.
Hence it makes sense to talk about the {\em dimension of the $k$th-order factor analysis model}, without specifying if we are dealing with cumulant or moment tensors.

\begin{remark}
One may also consider the Zariski closure of $\sC^{(\le k)}_{p,m}$, which is an algebraic variety. This leads to another fundamental question, that is studying the implicitization problem of the $k$th-order factor analysis model. There have been many attempts to study the polynomial relations describing a factor analysis model, and the problem is hard even when $k=2$. See \cite{harman1976modern,bekker1987rank,drton2007algebraic} for more details. The best known polynomial invariants are the \textit{tetrads}, which arise in one-factor analysis model. Additionally, the \textit{pentads}, which are fifth degree polynomials, were found to be model invariants for two-factor analysis model. When the number of factors is greater than two, there is no significant progress in finding the model invariants due to some computational difficulties.
The implicitization problem has also been recently studied for more general linear non-Gaussian graphical models \cite{Robeva_2021,amendola2021third}.
\end{remark}

\section{Main results}\label{sec: dimension higher order model}

In this section, we prove the two main results of this paper. First, in Theorem \ref{thm: dimension k order factor analysis model} we compute the dimension of the $k$th-order factor analysis model for all $k\ge 2$ and $p\ge m+1$. Secondly, in Theorem \ref{thm: when codimension k factor analysis model is positive} we study when the $k$th-order factor analysis model has positive codimension.

\begin{theorem}\label{thm: dimension k order factor analysis model}
Assume $k\ge 2$ and $p\ge m+1$.
Let $d_k=\dim(\sM^{(\le k)}_{p,m})=\dim(\sC^{(\le k)}_{p,m})$ be the dimension of the $k$th-order factor analysis model. Then
\begin{equation}\label{eq: dim k}
d_k = (k-1)p+(k-2)m+\min\left\{pm-\binom{m}{2},\binom{p+k-1}{k}-p\right\}\,.
\end{equation}
\end{theorem}

\begin{proof}
Consider the $k$th-order cumulant factor analysis model $\sC^{(\le k)}_{p,m}$. Our goal is to compute the dimension of $\sC^{(\le k)}_{p,m}$. Then, since there is a polynomial bijective correspondence between $\sC^{(\le k)}_{p,m}$ and $\sM^{(\le k)}_{p,m}$, described by the equations \eqref{eq: def cumulants} and \eqref{eq: moments from cumulants}, we conclude that $\dim(\sM^{(\le k)}_{p,m})=\dim(\sC^{(\le k)}_{p,m})$.
Up to an orthogonal transformation, we can assume that $\Lambda\in\R^{p\times m}$ is lower-triangular, namely
\[
\Lambda\in L_{p,m}=\{\Lambda\in\R^{p\times m}\mid\text{$\lambda_{ij}=0$ if $i<j$}\}\,.
\]
Since we are assuming $p\ge m+1$, we have that $\dim(L_{p,m})=pm-\binom{m}{2}$.
Recall that we identify the subset $\Delta^r(\R^n)$ with $\R^n$ for all $r\ge 2$.
Consider the polynomial map
\begin{equation}\label{eq: def phi k}
\begin{gathered}
\varphi_k\colon(\R^p)^{\times(k-1)}\times(\R^m)^{\times(k-1)}\times L_{p,m}\longrightarrow\prod_{r=2}^k\mathrm{Sym}^r(p)\\
\varphi_k(\sE^{(2)},\ldots,\sE^{(k)},\sD^{(2)},\ldots,\sD^{(k)},\Lambda)\coloneqq (\sT^{(r)})_{r=2}^k\,,
\end{gathered}
\end{equation}
where
\begin{equation}
\sT^{(r)}\coloneqq\sE^{(r)}+\sD^{(r)}\bullet^r\Lambda^T\quad\forall\,r\in\{2,\ldots,k\}\,.
\end{equation}
The dimensions of the domain and of the image space of $\varphi_k$ are respectively
\begin{align}\label{eq: write M, N arbitrary k}
\begin{split}
M &= \dim((\R^p)^{\times(k-1)}\times(\R^m)^{\times(k-1)}\times L_{p,m}) = (k-1)p+(k-1)m+pm-\binom{m}{2}\\
N &= \dim\left(\prod_{r=2}^k\mathrm{Sym}^r(p)\right) = \sum_{r=2}^k\binom{p+r-1}{r}\\
&= \sum_{r=0}^k\binom{p+r-1}{r}-p-1=\binom{p+k}{k}-p-1\,.
\end{split}
\end{align}
The last equality in \eqref{eq: write M, N arbitrary k} follows because we applied for $a=k$ and $b=p-1$ the identity
\[
\sum_{r=0}^a\binom{b+r}{r}=\binom{a+b+1}{a}\quad\forall\,a,b\ge 0\,.
\]
The core of the proof is studying the dimension of the image of $\varphi_k$, which is equal to the rank of the Jacobian matrix $J(\varphi_k)\in\R^{N\times M}$ evaluated at a generic point in the domain of $\varphi_k$.
In fact, by Definition \ref{def: k order cumulant and moment factor analysis model} and Proposition \ref{pro: parametrize k order cumulant factor analysis model}, this is equal to computing the dimension of $\sC^{(\le k)}_{p,m}$.

In order to study the rank of $J(\varphi_k)$, we choose coordinates for the domain and for the image space of $\varphi_k$. First, we identify the diagonal tensors $\sE^{(r)}=(\varepsilon^{(r)}_{j\cdots j})$, and $\sD^{(r)}=(\delta^{(r)}_{j\cdots j})$ with the vectors
\begin{align*}
\varepsilon^{(r)}&\coloneqq(\varepsilon^{(r)}_{1},\ldots,\varepsilon^{(r)}_{p})\in\R^{p}\,,\ \varepsilon^{(r)}_{j}\coloneqq\varepsilon^{(r)}_{j\cdots j}\ \forall\,j\in[p]\\
\delta^{(r)}&\coloneqq(\delta^{(r)}_{1},\ldots,\delta^{(r)}_{m})\in\R^{m}\,,\ \delta^{(r)}_{j}\coloneqq \delta^{(r)}_{j\cdots j}\ \forall\,j\in[m]\,,
\end{align*}
respectively. Secondly, we order the variables of the matrix $\Lambda$ as in the vector
\begin{align}\label{eq: notations lambda}
\begin{split}
\lambda & \coloneqq (\lambda_{\ge\,1}^T\mid\cdots\mid\lambda_{\ge\,m}^T)\in \R^{pm-\binom{m}{2}}\\
\lambda_{\ge\,j} & \coloneqq (\lambda_{j,j},\lambda_{j+1,j},\cdots,\lambda_{p,j})^T\in\R^{p-j+1}\ \forall\,j\in[m]\,.
\end{split}
\end{align}
Finally we denote by $t_{j_1\cdots j_r}^{(r)}$ the entries of the $r$th component $\sT^{(r)}$ of $\varphi_k$, and we consider only multiindices $j_1\le\cdots\le j_k$ due to the symmetry of $\sT^{(r)}$.
Then, using equation \eqref{eq: Tucker product used in the main proof} and the definition of $\varphi_r$ in \eqref{eq: def phi k}, we get
\begin{equation}\label{eq: write explicitly coords T}
t^{(r)}_{j_1\cdots j_r} =
\begin{cases}
\varepsilon^{(r)}_{j}+\sum_{\ell=1}^{\min\{j,m\}}\delta^{(r)}_{\ell}\lambda_{j\ell}^r & \text{if $j_1=\cdots=j_r=j$}\\
\sum_{\ell=1}^{\min\{j_1,m\}}\delta^{(r)}_{\ell}\lambda_{j_1\ell}\cdots\lambda_{j_r\ell} & \text{if $j_1\le\cdots\le j_r$ and $j_u\neq j_v$ for some $u,v$}\,.
\end{cases}
\end{equation}
Furthermore, we adopt also the following notations to simplify the writing of $J(\varphi_k)$:
\begin{align}\label{eq: notations tensors arbitrary k}
\begin{split}
t_{\Delta}^{(r)} & \coloneqq(t^{(r)}_{1},\ldots, t^{(r)}_{p})^T\in \R^p\,,\ t^{(r)}_{j}\coloneqq t^{(r)}_{j\cdots j}\ \forall\,j\in[p]\\
t_{<}^{(r)} & \coloneqq(t^{(r)}_{j_1\cdots j_r}\mid j_u\neq j_v\ \text{for some}\ u,v)^T\in \R^{\binom{p+r-1}{r}-p}\,.
\end{split}
\end{align}
The Jacobian matrix $J(\varphi_k)$ can be written as
\begin{equation}\label{eq: block structure jacobian phi_k}
J(\varphi_k) =
\begin{cases}
\begin{pNiceMatrix}[first-row,first-col,margin]
& \varepsilon^{(2)} & \delta^{(2)} & \lambda \\
t_{\Delta}^{(2)} & \Block[fill=gray!10,draw]{1-1}{I_p} & \Block[fill=gray!10]{1-1}{} & \Block[fill=gray!10]{1-1}{} \\
t_{<}^{(2)} & & \Block[fill=gray!10,draw]{1-1}{A^{(2)}} & \Block[fill=gray!10,draw]{1-1}{B^{(2)}}\\
\end{pNiceMatrix} & \text{if $k=2$}\\
&\\
\begin{pNiceMatrix}[first-row,first-col,margin]
& \varepsilon^{(2)} & \cdots & \varepsilon^{(k)}& \delta^{(2)} & \cdots & \delta^{(k-1)} & \delta^{(k)} & \lambda \\
t_{\Delta}^{(2)} & \Block[fill=gray!10,draw]{1-1}{I_p} & \Block[fill=gray!10]{1-7}{} \\
\vdots & & \Block[fill=gray!10,draw]{1-1}{\ddots} & \Block[fill=gray!10]{1-6}{} \\
t_{\Delta}^{(k)} & & & \Block[fill=gray!10,draw]{1-1}{I_p} & \Block[fill=gray!10]{1-5}{} \\
t_{<}^{(2)} & & & & \Block[fill=gray!10,draw]{1-1}{A^{(2)}} & \Block[fill=gray!10]{1-1}{} & \Block[fill=gray!10]{1-1}{} & \Block[fill=gray!10]{1-1}{} & \Block[fill=gray!10,draw]{1-1}{B^{(2)}} \\
\vdots & & & & & \Block[fill=gray!10,draw]{1-1}{\ddots} & \Block[fill=gray!10]{1-1}{} & \Block[fill=gray!10]{1-1}{} & \Block[fill=gray!10,draw]{1-1}{\vdots} \\
t_{<}^{(k-1)} & & & & & & \Block[fill=gray!10,draw]{1-1}{A^{(k-1)}} & \Block[fill=gray!10]{1-1}{} & \Block[fill=gray!10,draw]{1-1}{B^{(k-1)}} \\
t_{<}^{(k)} & & & & & & & \Block[fill=gray!10,draw]{1-1}{A^{(k)}} & \Block[fill=gray!10,draw]{1-1}{B^{(k)}}
\end{pNiceMatrix} & \text{if $k\ge 3$}
\end{cases}
\end{equation}
where the white part is zero and for all $r$
\begin{equation}\label{eq: blocks Ek, Ak}
A^{(r)}=
\frac{\partial\,t^{(r)}_{<}}{\partial\,\delta^{(r)}} \in \R^{\left(\binom{p+r-1}{r}-p\right)\times m}\,,\quad B^{(r)}=
\frac{\partial\,t^{(r)}_{<}}{\partial\,\lambda} \in \R^{\left(\binom{p+r-1}{r}-p\right)\times\left(pm-\binom{m}{2}\right)}\,.
\end{equation}
Thanks to the upper triangular block structure of $J(\varphi_k)$ shown in \eqref{eq: block structure jacobian phi_k}, we conclude that
\begin{equation}\label{eq: rank jacobian sum any r}
\mathrm{rank}(J(\varphi_k)) \ge
\begin{cases}
p+\mathrm{rank}([A^{(2)}\mid B^{(2)}]) & \text{if $k=2$}\\
(k-1)p+\sum_{r=2}^{k-1}\mathrm{rank}(A^{(r)})+\mathrm{rank}([A^{(k)}\mid B^{(k)}]) & \text{if $k\ge 3$}\,.
\end{cases}
\end{equation}

\noindent{\bf Claim 1.} For all $r\in\{2,\ldots,k\}$, we have that
\begin{equation}
    \mathrm{rank}([A^{(r)}\mid B^{(r)}]) = \mathrm{rank}(B^{(r)})\,.
\end{equation}
Given $r\in\{2,\ldots,k\}$ and $i\in[m]$, for every multi-index $(j_1,\ldots,j_r)$, we have that
\begin{equation}\label{eq: partial T wrt deltas Euler identity}
\frac{\partial\,t_{j_1\cdots j_r}^{(r)}}{\partial\,\delta_i^{(r)}} = \lambda_{j_1i}\cdots\lambda_{j_ri} = \frac{1}{r}\sum_{s=1}^p\lambda_{si}\frac{\partial\,(\lambda_{j_1i}\cdots\lambda_{j_ri})}{\partial\,\lambda_{si}} = \frac{1}{r}\sum_{s=1}^p\frac{\lambda_{si}}{\delta_i^{(r)}}\frac{\partial\,t_{j_1\cdots j_r}^{(r)}}{\partial\,\lambda_{si}}\,,
\end{equation}
where the second equality follows from Euler's homogeneous function theorem and the last identity holds since, for every pair of indices $(u,v)$, by \eqref{eq: write explicitly coords T} we have that
\begin{equation}\label{eq: derivative wrt lambda}
\frac{\partial\,t_{j_1\cdots j_r}^{(r)}}{\partial\,\lambda_{uv}} = \frac{\partial\,(\delta_v^{(r)}\lambda_{j_1v}\cdots\lambda_{j_rv})}{\partial\,\lambda_{uv}} = \delta_v^{(r)}\alpha_u\lambda_{uv}^{\alpha_u-1}\prod_{\ell\mid j_\ell\neq u}\lambda_{j_\ell v}\,,
\end{equation}
where $\alpha_u(j_1,\ldots,j_r)\coloneqq|\{\ell\mid j_\ell=u\}|$.
Equation \eqref{eq: partial T wrt deltas Euler identity} implies that every column $\partial\,t_<^{(r)}/\partial\,\delta_i^{(r)}$ of $A^{(r)}$ is a linear combination of the columns $\partial\,t_<^{(r)}/\partial\,\lambda_{si}$ of $B^{(r)}$, proving Claim 1.

\noindent{\bf Claim 2.} For all $r\in\{2,\ldots,k\}$ and for a generic choice of parameters in the domain of $\varphi_k$, we have
\begin{equation}\label{eq: rank D r}
    \mathrm{rank}(A^{(r)})=\min\left\{m,\binom{p+r-1}{r}-p\right\}=m\,.
\end{equation}

By construction of $\Lambda$ and by the first equality in \eqref{eq: partial T wrt deltas Euler identity}, we get that $\partial\,t_{j_1\cdots j_r}^{(r)}/\partial\,\delta_i^{(r)}=0$ when $j_1<i$. Define
\begin{equation}\label{eq: def t j star}
t_{j,*}^{(r)} \coloneqq (t_{j_1\cdots j_r}^{(r)}\mid\text{$j_1=j$ and $j_u\neq j_v$ for some $u,v$})\quad\forall j\in[p-1]\,.
\end{equation}
In particular, each vector $t_{j,*}^{(r)}$ has $\binom{p-j+r-1}{r-1}-1$ entries. Reordering the rows of $A^{(r)}$ using the concatenation of the vectors $t_{1,*}^{(r)}$,\ldots,$t_{p-1,*}^{(r)}$, we obtain a matrix like
\begin{equation}\label{eq: block structure D_r}
D^{(r)} =
\begingroup
\renewcommand*{\arraystretch}{2}
\left(
\begin{array}{cccc|ccc}
\frac{\partial\,t_{1,*}^{(r)}}{\partial\,\delta_1^{(r)}} & \frac{\partial\,t_{2,*}^{(r)}}{\partial\,\delta_1^{(r)}} & \cdots & \frac{\partial\,t_{m,*}^{(r)}}{\partial\,\delta_1^{(r)}} & \frac{\partial\,t_{m+1,*}^{(r)}}{\partial\,\delta_1^{(r)}} & \cdots & \frac{\partial\,t_{p-1,*}^{(r)}}{\partial\,\delta_1^{(r)}}\\
 & \frac{\partial\,t_{2,*}^{(r)}}{\partial\,\delta_2^{(r)}} & \cdots & \frac{\partial\,t_{m,*}^{(r)}}{\partial\,\delta_2^{(r)}} & \frac{\partial\,t_{m+1,*}^{(r)}}{\partial\,\delta_2^{(r)}} & \cdots & \frac{\partial\,t_{p-1,*}^{(r)}}{\partial\,\delta_2^{(r)}}\\
 & & \ddots & \vdots & \vdots & & \vdots\\
 & & & \frac{\partial\,t_{m,*}^{(r)}}{\partial\,\delta_m^{(r)}} & \frac{\partial\,t_{m+1,*}^{(r)}}{\partial\,\delta_m^{(r)}} & \cdots & \frac{\partial\,t_{p-1,*}^{(r)}}{\partial\,\delta_m^{(r)}}
\end{array}
\right)^T\,,
\endgroup
\end{equation}
where the void blocks are identically zero, and the right block is not present if $p=m+1$.
Since each entry of the block column vectors $\partial\,t_{i,*}^{(r)}/\partial\,\delta_i^{(r)}$ is a monomial in the variables $\lambda_{uv}$, the matrix $A^{(r)}$ has full rank $\min\left\{m,\binom{p+r-1}{r}-p\right\}$ if $\lambda_{uv}\neq 0$ for all $u\ge v$. The last identity in \eqref{eq: rank D r} follows by the assumption $p\ge m+1$.

\noindent{\bf Claim 3.} For all $k\ge 2$ and for a generic choice of parameters in the domain of $\varphi_k$, we have
\begin{equation}\label{eq: rank L^k}
    \mathrm{rank}(B^{(k)})=\min\left\{pm-\binom{m}{2},\binom{p+k-1}{k}-p\right\}\,.
\end{equation}

The case $k=2$ is covered in the proof of \cite[Theorem 2]{drton2007algebraic}, where the matrix $B^{(2)}$ is called $A$ and has full rank equal to $\min\{pm-\binom{m}{2},\binom{p}{2}\}$. Therefore, we now prove Claim 3 for $k\ge 3$.
By construction of $\Lambda$ and by the first equality in \eqref{eq: partial T wrt deltas Euler identity}, for all $u\ge v$, the partial derivative $\partial\,t_{j_1\cdots j_r}^{(k)}/\partial\,\lambda_{uv}^{(k)}$ is identically zero if $j_1<v$. Using the notations introduced in \eqref{eq: notations lambda} and \eqref{eq: def t j star}, we can reorder the rows and the columns of $B^{(k)}$ and obtain a matrix like
\begin{equation}\label{eq: block structure L_r}
B^{(k)} =
\begingroup
\renewcommand*{\arraystretch}{2}
\left(
\begin{array}{cccc|ccc}
\frac{\partial\,t_{1,*}^{(k)}}{\partial\,\lambda_{\ge\,1}} & \frac{\partial\,t_{2,*}^{(k)}}{\partial\,\lambda_{\ge\,1}} & \cdots & \frac{\partial\,t_{m,*}^{(k)}}{\partial\,\lambda_{\ge\,1}} & \frac{\partial\,t_{m+1,*}^{(k)}}{\partial\,\lambda_{\ge\,1}} & \cdots & \frac{\partial\,t_{p-1,*}^{(k)}}{\partial\,\lambda_{\ge\,1}}\\
 & \frac{\partial\,t_{2,*}^{(k)}}{\partial\,\lambda_{\ge\,2}} & \cdots & \frac{\partial\,t_{m,*}^{(k)}}{\partial\,\lambda_{\ge\,2}} & \frac{\partial\,t_{m+1,*}^{(k)}}{\partial\,\lambda_{\ge\,2}} & \cdots & \frac{\partial\,t_{p-1,*}^{(k)}}{\partial\,\lambda_{\ge\,2}}\\
 & & \ddots & \vdots & \vdots & \vdots \\
 & & & \frac{\partial\,t_{m,*}^{(k)}}{\partial\,\lambda_{\ge\,m}} & \frac{\partial\,t_{m+1,*}^{(k)}}{\partial\,\lambda_{\ge\,m}} & \cdots & \frac{\partial\,t_{p-1,*}^{(k)}}{\partial\,\lambda_{\ge\,m}}
\end{array}
\right)^T\,,
\endgroup
\end{equation}
where the void blocks are identically zero, and the right block is not present if $p=m+1$. Using the lower triangular block structure of $B^{(k)}$, we conclude that $\mathrm{rank}(B^{(k)}) \ge \sum_{i=1}^m\mathrm{rank}(\partial\,t_{i,*}^{(k)}/\partial\,\lambda_{\ge\,i})$.
If we show that every block matrix $\partial\,t_{i,*}^{(k)}/\partial\,\lambda_{\ge\,i}$ has full rank for a generic choice of parameters, then $B^{(k)}$ has full rank given by \eqref{eq: rank L^k}.

First, we assume $\delta_s^{(k)}\neq 0$ for all $s\in[m]$, and from now on we set $\delta_s^{(k)}=1$ for all $s\in[m]$ for simplicity.
Every matrix $\partial\,t_{i,*}^{(k)}/\partial\,\lambda_{\ge\,i}$ has $\binom{p-i+k-1}{k-1}-1$ rows and $p-i+1$ columns, see the following equation \eqref{eq: example r=4} for an explicit example when $k=4$, $p=5$, $m=2$, and $i=1$:
\begin{equation}\label{eq: example r=4}
\frac{\partial\,t_{1,*}^{(4)}}{\partial\,\lambda_{\ge\,1}}=
\begingroup
\renewcommand*{\arraystretch}{1.3}
\begin{scriptsize}
\begin{pNiceMatrix}[first-row,first-col,margin]
& \lambda_{11} & \lambda_{21} & \lambda_{31} & \lambda_{41} & \lambda_{51} \\
t_{1112}^{(4)} & \boldsymbol{3\,\lambda_{11}^{2}\lambda_{21}} & \boldsymbol{\lambda_{11}^{3}}&0&0&0\\
t_{1113}^{(4)} & \boldsymbol{3\,\lambda_{11}^{2}\lambda_{31}}&0&\boldsymbol{\lambda_{11}^{3}}&0&0\\
t_{1114}^{(4)} & \boldsymbol{3\,\lambda_{11}^{2}\lambda_{41}}&0&0&\boldsymbol{\lambda_{11}^{3}}&0\\
t_{1115}^{(4)} & \boldsymbol{3\,\lambda_{11}^{2}\lambda_{51}}&0&0&0&{\lambda_{11}^{3}}\\
t_{1122}^{(4)} & \boldsymbol{2\,\lambda_{11}\lambda_{21}^{2}}&\boldsymbol{2\,\lambda_{11}^{2}\lambda_{21}}&0&0&0\\
t_{1123}^{(4)} & 2\,\lambda_{11}\lambda_{21}\lambda_{31}&\lambda_{11}^{2}\lambda_{31}&\lambda_{11}^{2}\lambda_{21}&0&0\\
t_{1124}^{(4)} & 2\,\lambda_{11}\lambda_{21}\lambda_{41}&\lambda_{11}^{2}\lambda_{41}&0&\lambda_{11}^{2}\lambda_{21}&0\\
t_{1125}^{(4)} & 2\,\lambda_{11}\lambda_{21}\lambda_{51}&\lambda_{11}^{2}\lambda_{51}&0&0&\lambda_{11}^{2}\lambda_{21}\\
t_{1133}^{(4)} & \boldsymbol{2\,\lambda_{11}\lambda_{31}^{2}}&0&\boldsymbol{2\,\lambda_{11}^{2}\lambda_{31}}&0&0\\
t_{1134}^{(4)} & 2\,\lambda_{11}\lambda_{31}\lambda_{41}&0&\lambda_{11}^{2}\lambda_{41}&\lambda_{11}^{2}\lambda_{31}&0\\
t_{1135}^{(4)} & 2\,\lambda_{11}\lambda_{31}\lambda_{51}&0&\lambda_{11}^{2}\lambda_{51}&0&\lambda_{11}^{2}\lambda_{31}\\
t_{1144}^{(4)} & \boldsymbol{2\,\lambda_{11}\lambda_{41}^{2}}&0&0&\boldsymbol{2\,\lambda_{11}^{2}\lambda_{41}}&0\\
t_{1145}^{(4)} & 2\,\lambda_{11}\lambda_{41}\lambda_{51}&0&0&\lambda_{11}^{2}\lambda_{51}&\lambda_{11}^{2}\lambda_{41}\\
t_{1155}^{(4)} & \boldsymbol{2\,\lambda_{11}\lambda_{51}^{2}}&0&0&0&\boldsymbol{2\,\lambda_{11}^{2}\lambda_{51}}\\
t_{1222}^{(4)} & \boldsymbol{\lambda_{21}^{3}}&\boldsymbol{3\,\lambda_{11}\lambda_{21}^{2}}&0&0&0\\
t_{1223}^{(4)} & \lambda_{21}^{2}\lambda_{31}&2\,\lambda_{11}\lambda_{21}\lambda_{31}&\lambda_{11}\lambda_{21}^{2}&0&0\\
t_{1224}^{(4)} & \lambda_{21}^{2}\lambda_{41}&2\,\lambda_{11}\lambda_{21}\lambda_{41}&0&\lambda_{11}\lambda_{21}^{2}&0\\
t_{1225}^{(4)} & \lambda_{21}^{2}\lambda_{51}&2\,\lambda_{11}\lambda_{21}\lambda_{51}&0&0&\lambda_{11}\lambda_{21}^{2}\\
t_{1233}^{(4)} & \lambda_{21}\lambda_{31}^{2}&\lambda_{11}\lambda_{31}^{2}&2\,\lambda_{11}\lambda_{21}\lambda_{31}&0&0\\
t_{1234}^{(4)} & \lambda_{21}\lambda_{31}\lambda_{41}&\lambda_{11}\lambda_{31}\lambda_{41}&\lambda_{11}\lambda_{21}\lambda_{41}&\lambda_{11}\lambda_{21}\lambda_{31}&0\\
t_{1235}^{(4)} & \lambda_{21}\lambda_{31}\lambda_{51}&\lambda_{11}\lambda_{31}\lambda_{51}&\lambda_{11}\lambda_{21}\lambda_{51}&0&\lambda_{11}\lambda_{21}\lambda_{31}\\
t_{1244}^{(4)} & \lambda_{21}\lambda_{41}^{2}&\lambda_{11}\lambda_{41}^{2}&0&2\,\lambda_{11}\lambda_{21}\lambda_{41}&0\\
t_{1245}^{(4)} & \lambda_{21}\lambda_{41}\lambda_{51}&\lambda_{11}\lambda_{41}\lambda_{51}&0&\lambda_{11}\lambda_{21}\lambda_{51}&\lambda_{11}\lambda_{21}\lambda_{41}\\
t_{1255}^{(4)} & \lambda_{21}\lambda_{51}^{2}&\lambda_{11}\lambda_{51}^{2}&0&0&2\,\lambda_{11}\lambda_{21}\lambda_{51}\\
t_{1333}^{(4)} & \boldsymbol{\lambda_{31}^{3}}&0&\boldsymbol{3\,\lambda_{11}\lambda_{31}^{2}}&0&0\\
t_{1334}^{(4)} & \lambda_{31}^{2}\lambda_{41}&0&2\,\lambda_{11}\lambda_{31}\lambda_{41}&\lambda_{11}\lambda_{31}^{2}&0\\
t_{1335}^{(4)} & \lambda_{31}^{2}\lambda_{51}&0&2\,\lambda_{11}\lambda_{31}\lambda_{51}&0&\lambda_{11}\lambda_{31}^{2}\\
t_{1344}^{(4)} & \lambda_{31}\lambda_{41}^{2}&0&\lambda_{11}\lambda_{41}^{2}&2\,\lambda_{11}\lambda_{31}\lambda_{41}&0\\
t_{1345}^{(4)} & \lambda_{31}\lambda_{41}\lambda_{51}&0&\lambda_{11}\lambda_{41}\lambda_{51}&\lambda_{11}\lambda_{31}\lambda_{51}&\lambda_{11}\lambda_{31}\lambda_{41}\\
t_{1355}^{(4)} & \lambda_{31}\lambda_{51}^{2}&0&\lambda_{11}\lambda_{51}^{2}&0&2\,\lambda_{11}\lambda_{31}\lambda_{51}\\
t_{1444}^{(4)} & \boldsymbol{\lambda_{41}^{3}}&0&0&\boldsymbol{3\,\lambda_{11}\lambda_{41}^{2}}&0\\
t_{1445}^{(4)} & \lambda_{41}^{2}\lambda_{51}&0&0&2\,\lambda_{11}\lambda_{41}\lambda_{51}&\lambda_{11}\lambda_{41}^{2}\\
t_{1455}^{(4)} & \lambda_{41}\lambda_{51}^{2}&0&0&\lambda_{11}\lambda_{51}^{2}&2\,\lambda_{11}\lambda_{41}\lambda_{51}\\
t_{1555}^{(4)} & \boldsymbol{\lambda_{51}^{3}}&0&0&0&\boldsymbol{3\,\lambda_{11}\lambda_{51}^{2}}
\end{pNiceMatrix}\,.
\end{scriptsize}
\endgroup
\end{equation}
Fix an index $s\in\{i+1,\ldots,p\}$. We want to find all non-decreasing tuples $(j_2,\ldots,j_k)$ with $j_2\ge i$ such that $\partial\,t_{i,j_2,\ldots,j_k}^{(k)}/\partial\,\lambda_{ji}$ is not identically zero if and only if $j=s$. The latter condition implies that all indices $j_2,\ldots, j_k$ are in $\{i,s\}$, and at least one is equal to $s$. This means that there are exactly $k-1$ such choices of multiindices. In the example highlighted in \eqref{eq: example r=4}, for each $s\in\{2,3,4,5\}$, there are exactly three triples $(j_2,j_3,j_4)$ such that in the row labeled by $t_{1j_2j_3j_4}^{(4)}$, the only non-zero element (apart from the one in the first column) is in the $s$th column, in particular $(j_2,j_3,j_4)\in\{(i,i,s),(i,s,s),(s,s,s)\}$.
In total, we are selecting $(k-1)(p-i)$ rows of the matrix $\partial\,t_{i,*}^{(r)}/\partial\,\lambda_{\ge\,i}$, and $(k-1)(p-i)$ is at least the number of columns $p-i+1$ if and only if $(k-2)(p-i)\ge 1$. The latter inequality is always satisfied when $k\ge 3$, since $i\le p-1$. Call $R_i^{(k)}$ the submatrix of $\partial\,t_{i,*}^{(k)}/\partial\,\lambda_{\ge\,i}$ just constructed. For example, from the matrix in \eqref{eq: example r=4}, we extract the entries in bold, thus getting the submatrix
\begin{equation}
R_1^{(4)}=
\begingroup
\renewcommand*{\arraystretch}{1.3}
\begin{small}
\begin{pNiceMatrix}[first-row,first-col,margin]
& \lambda_{11} & \lambda_{21} & \lambda_{31} & \lambda_{41} & \lambda_{51}\\
t_{1112}^{(4)} & 3\,\lambda_{11}^{2}\lambda_{21}&\lambda_{11}^{3}&0&0&0\\
t_{1113}^{(4)} & 3\,\lambda_{11}^{2}\lambda_{31}&0&\lambda_{11}^{3}&0&0\\
t_{1114}^{(4)} & 3\,\lambda_{11}^{2}\lambda_{41}&0&0&\lambda_{11}^{3}&0\\
t_{1115}^{(4)} & 3\,\lambda_{11}^{2}\lambda_{51}&0&0&0&\lambda_{11}^{3}\\
t_{1122}^{(4)} & 2\,\lambda_{11}\lambda_{21}^{2}&2\,\lambda_{11}^{2}\lambda_{21}&0&0&0\\
t_{1133}^{(4)} & 2\,\lambda_{11}\lambda_{31}^{2}&0&2\,\lambda_{11}^{2}\lambda_{31}&0&0\\
t_{1144}^{(4)} & 2\,\lambda_{11}\lambda_{41}^{2}&0&0&2\,\lambda_{11}^{2}\lambda_{41}&0\\
t_{1155}^{(4)} & 2\,\lambda_{11}\lambda_{51}^{2}&0&0&0&2\,\lambda_{11}^{2}\lambda_{51}\\
t_{1222}^{(4)} & \lambda_{21}^{3}&3\,\lambda_{11}\lambda_{21}^{2}&0&0&0\\
t_{1333}^{(4)} & \lambda_{31}^{3}&0&3\,\lambda_{11}\lambda_{31}^{2}&0&0\\
t_{1444}^{(4)} & \lambda_{41}^{3}&0&0&3\,\lambda_{11}\lambda_{41}^{2}&0\\
t_{1555}^{(4)} & \lambda_{51}^{3}&0&0&0&3\,\lambda_{11}\lambda_{51}^{2}
\end{pNiceMatrix}\,.
\end{small}
\endgroup
\end{equation}
It remains to show that the matrix $R_i^{(k)}$ has full rank $p-i+1$ for a generic choice of $\Lambda$.
Up to reordering the indices, we can assume that the first $p-i+1$ rows of the vector $t_{i,*}^{(k)}$ are $t_{i\cdots i,i+1}^{(k)},\ldots,t_{i\cdots i,p}^{(k)},t_{i\cdots i,i+1,i+1}^{(k)}$. Using \eqref{eq: derivative wrt lambda}, the top maximal minor of $R_i^{(k)}$ is equal to
\begin{equation}\label{eq: max minor R^k}
\begin{gathered}
\begingroup
\renewcommand*{\arraystretch}{1.3}
\left|
\begin{array}{c|ccc}
(k-1)\,\lambda_{ii}^{k-2}\lambda_{i+1,i} & \lambda_{ii}^{k-1} & & 0 \\
\vdots& &\ddots& \\
(k-1)\,\lambda_{ii}^{k-2}\lambda_{p,i} & 0 & & \lambda_{ii}^{k-1} \\\hline
(k-2)\,\lambda_{ii}^{k-3}\lambda_{i+1,i}^2 & 2\,\lambda_{ii}^{k-2}\lambda_{i+1,i} & \cdots & 0
\end{array}
\right|=
\endgroup
\\
= (-1)^{p-i+2}\,\lambda_{ii}^{(p-i-1)(k-1)}[(k-2)\,\lambda_{ii}^{k-3}\lambda_{ii}^{k-1}\lambda_{i+1,i}^2-2\,\lambda_{ii}^{k-2}\lambda_{i+1,i}(k-1)\,\lambda_{ii}^{k-2}\lambda_{i+1,i}]\\
= (-1)^{p-i+1}(k-1)\,\lambda_{ii}^{(p-i+1)(k-1)-2}\lambda_{i+1,i}^2\,,
\end{gathered}
\end{equation}
where in the second equality in \eqref{eq: max minor R^k} we have used the Laplace expansion along the last row.
Since the previous minor is a monomial in the $\lambda_{uv}$'s, it is nonzero if all variables $\lambda_{uv}$ are nonzero. This concludes the proof of Claim 3.

Summing up, Claims 1,2,3 tell us that, for a generic choice of parameters $\sE^{(2)}$,\ldots, $\sE^{(k)}$, $\sD^{(2)}$,\ldots, $\sD^{(k)}$, and $\Lambda$, the diagonal blocks $A^{(2)}$,\ldots,$A^{(k-1)}$, and $[A^{(k)}\mid B^{(k)}]$ have full rank. This implies that equality holds in \eqref{eq: rank jacobian sum any r}, in particular for $k=2$
\begin{equation}\label{eq: rank jacobian sum any r part 2 k=2}
\mathrm{rank}(J(\varphi_2)) = p+\mathrm{rank}(B^{(2)})
= p+\min\left\{pm-\binom{m}{2},\binom{p+1}{2}-p\right\}\,,
\end{equation}
whereas for $k\ge 3$
\begin{align}\label{eq: rank jacobian sum any r part 2 k>=3}
\begin{split}
\mathrm{rank}(J(\varphi_k)) &= (k-1)p+\sum_{r=2}^{k-1}\mathrm{rank}(A^{(r)})+\mathrm{rank}(B^{(k)})\\
&= (k-1)p+(k-2)m+\min\left\{pm-\binom{m}{2},\binom{p+k-1}{k}-p\right\}\,.
\end{split}
\end{align}
This yields the dimension formula in \eqref{eq: dim k} and the proof is complete.
\end{proof}

\begin{remark}
Note that the model $\sC^{(\le k)}_{p,m}$ has the same dimension of the image of the restriction $\varphi_k|_{\{\sD^{(2)}=I_m\}}$ for all $k\ge 2$. Indeed, without loss of generality, we can assume that the covariance matrix of the vector $Y$ of hidden variables is the identity matrix of size $m$. 
However, the rank of the Jacobian matrix of $\varphi_k|_{\{\sD^{(2)}=I_m\}}$ is equal to the rank of $J(\varphi_k)$. In fact, we have shown in Claim 1 of the proof of Theorem \ref{thm: dimension k order factor analysis model} that the columns corresponding to the vector of parameters $\delta^{(2)}$ are in the span of the columns corresponding to the vector of parameters $\lambda$.
\end{remark}

\begin{remark}\label{rmk: simplified formula d_k}
Consider the dimension formula \eqref{eq: dim k}. We show that
\begin{equation}\label{eq: inequality codimension any k>=3}
pm-\binom{m}{2}\le\binom{p+k-1}{k}-p
\end{equation}
for all $k\ge 3$ and $p\ge m+1$. First, we observe that the sequence $\{\binom{p+k-1}{k}\}$ is nondecreasing in $k$, therefore it is sufficient to show that $pm-\binom{m}{2}\le\binom{p+2}{3}-p$ for all $p\ge m+1$.
The previous inequality is equivalent to $f(p)\coloneqq p^3+3p^2-2(3m+2)p+3m(m-1)\ge 0$. We have $f(m+1)=m(m^2+3m-4)\ge 0$ for all $m\ge 1$, and one verifies that the largest root of $f'(p)$ is $-1+\sqrt{2m+\frac{7}{3}}\le m+1$, therefore $f(p)\ge0$ for all $p\ge m+1$. This implies that \eqref{eq: inequality codimension any k>=3} holds for all $k\ge 3$ and $p\ge m+1$, hence in this case the dimension formula \eqref{eq: dim k} simplifies as
\begin{equation}\label{eq: simplified dim k}
d_k = (m+k-1)p+(k-2)m-\binom{m}{2}\,.    
\end{equation}
\end{remark}

\begin{theorem}\label{thm: when codimension k factor analysis model is positive}
Let $c_k$ be the codimension $\mathrm{codim}(\sM^{(\le k)}_{p,m})=\mathrm{codim}(\sC^{(\le k)}_{p,m})$ of the $k$th-order factor analysis model as a subvariety in $\prod_{r=2}^k\mathrm{Sym}^r(p)$.
\begin{enumerate}
    \item[$(i)$] If $k=2$, then $c_k>0$ if and only if $p\ge\left\lfloor m+\frac{1}{2}\sqrt{8m+1}+\frac{1}{2}\right\rfloor+1$.
    \item[$(ii)$] If $k\ge 3$ and $p\ge m+1$, then $c_k=h_m^{(k)}(p)/k!$, where
    \begin{equation}\label{eq: h_m k}
    h_m^{(k)}(p)=\prod_{i=1}^k(p+i)-k!(k+m)p+\frac{k!}{2}[m^2+(3-2k)m-2]\,.
    \end{equation}
    In particular
    \begin{enumerate}
    \item[$(a)$] if $m\in[2k-3]$, then $h_m^{(k)}(p)$ has a unique positive root $p^{(k)}$, therefore $c_k>0$ if $p\ge\lfloor p^{(k)}\rfloor+1$.
    \item[$(b)$] for finitely many values of $m\ge 2k-2$, the polynomial $h_m^{(k)}(p)$ has two positive roots, the largest denoted by $p^{(k)}$. Therefore $c_k>0$ if $p\ge\lfloor p^{(k)}\rfloor+1$.
    \item[$(c)$] There exists an integer $m^*\ge 2k-2$ such that $h_m^{(k)}(p)$ has no positive roots for $m\ge m^*$, in particular $c_k>0$ for all $m\ge m^*$ and $p\ge m+1$.
    \end{enumerate}
\end{enumerate}
\end{theorem}

Before starting the proof of Theorem \ref{thm: when codimension k factor analysis model is positive}, we recall a property of real univariate polynomials related to a theorem of P\'olya \cite{polya1904positive} (see also \cite[Theorem 56, p. 57]{hardy1952inequalities}). The following statement is taken from \cite[Exercise 3(a), Chapter 3]{stanley2013algebraic}.

\begin{proposition}\label{prop: property *}
Let $f\in\R[x]$ be a nonzero polynomial with real coefficients. The following two conditions are equivalent.
\begin{itemize}
\item[$(i)$] There exists a nonzero polynomial $g\in\R[x]$ with real coefficients such that all coefficients of $fg$ are nonnegative.
\item[$(ii)$] There does not exist a real number $a>0$ such that $f(a)=0$.
\end{itemize}
\end{proposition}

\begin{proof}[Proof of Theorem \ref{thm: when codimension k factor analysis model is positive}]
The case $k=2$ is studied in \cite[Theorem 2]{drton2007algebraic}, so we assume $k\ge 3$.
The codimension $c_k$ is the difference between the dimension $N=\binom{p+k}{k}-p-1$ of the image space $\prod_{r=2}^k\mathrm{Sym}^r(p)$ computed in \eqref{eq: write M, N arbitrary k} and the value $d_k$ computed in \eqref{eq: simplified dim k}. This yields the polynomial $h_m^{(k)}(p)$ in \eqref{eq: h_m k} up to division by $k!$.
Given $m$ and $k$, we want to find the minimum integer $p^{(k)}$ such that $h_m^{(k)}(p)>0$ for every $p\ge p^{(k)}$. In other words, we want to give sufficient conditions on $p$ such that $c_k>0$.
For all $j\in\{0,\ldots,k\}$, let
\begin{equation}
e_j(x_1,\ldots,x_k)\coloneqq
\begin{cases}
1 & \text{if $j=0$}\\
\sum_{1\le\ell_1<\cdots<\ell_j\le k}x_{\ell_1}\cdots x_{\ell_j} & \text{if $j\in[k]$}
\end{cases}
\end{equation}
be the $j$-th elementary symmetric polynomial on $x_1,\ldots,x_k$.
We also denote by $e_j^{(k)}$ the number $e_j(1,\ldots,k)$.
We can rewrite $h_m^{(k)}(p)$ as
\begin{align*}
h_m^{(k)}(p) &= \sum_{j=0}^k e_j^{(k)}p^{k-j}-k!(k+m)p+\frac{k!}{2}[m^2+(3-2k)m-2]\\
&= \sum_{j=0}^{k-2} e_j^{(k)}p^{k-j}+[e_{k-1}^{(k)}-k!(k+m)]p+\frac{k!}{2}m(m-2k+3)\,.
\end{align*}
In particular the coefficients of $p^\ell$ with $\ell\ge 2$ are always positive. Regarding the linear coefficient
\[
u(k,m)\coloneqq e_{k-1}^{(k)}-k!(k+m) = k!\left(\sum_{i=1}^k\frac{1}{i}-k-m\right)\,,
\]
it is not difficult to verify that $u(k,m)<0$ for all $m\ge 1$.
This implies that there is always at least one change of sign in $h_m^{(k)}(p)$, but no more than two.
More precisely, the number of sign changes depends on the sign of the constant term
\[
v(k,m)\coloneqq\frac{k!}{2}m(m-2k+3)
\]
of $h_m^{(k)}(p)$. In particular $v(k,m)>0$ for all $m\ge 2k-2$.
Therefore, applying Descartes' rule of signs, the polynomial $h_m^{(k)}(p)$ has exactly one positive root for $m\in[2k-3]$, and either two or zero real positive roots for $m\ge 2k-2$.
Since for $m\in[2k-3]$ the polynomial $h_m^{(k)}(p)$ has exactly one positive root, say $p^{(k)}$, then for $p\ge p^{(k)}$, the value of $h_m^{(k)}(p)$ is positive. This proves part $(a)$ of the statement.
To prove $(b)$ and $(c)$, it remains to show that $h_m^{(k)}$ has exactly two positive real roots for finitely many integers $m\ge 2k-2$ and then for sufficiently large $m$, it has no positive real roots.

\noindent First, suppose that $k\in\{3,4,5\}$. In particular, we have
\begin{align}
\begin{split}
    h_m^{(3)}(p) &= p^3+6p^2-(6m+7)p+3m(m-3)\\
    h_m^{(4)}(p) &= p^4+10p^3+35p^2-2(12m+23)p+12m(m-5)\\
    h_m^{(5)}(p) &= p^5+15p^4+85p^3+225p^2-2(60m+163)p+60m(m-7)
\end{split}
\end{align}
The polynomial $h_m^{(3)}(p)$ has one positive root for $m\in [3]$, two distinct positive roots for $m\in\{4,5,6\}$, and no positive root for all $m\ge 7$.
The polynomial $h_m^{(4)}$ has one positive root for $m\in[5]$, two distinct positive roots for $m\in\{6,7\}$, and no positive root for all $m\ge 8$.
Instead the polynomial $h_m^{(5)}$ has one positive root for $m\in[7]$, two distinct positive roots for $m\in\{8,9\}$, and no positive root for all $m\ge 10$.
The approximate values of the largest positive roots $p^{(k)}$ of $h_m^{(3)}$, $h_m^{(4)}$ and $h_m^{(5)}$ are displayed in Table \ref{tab: values p0 k}.
\begingroup
\renewcommand*{\arraystretch}{1.3}
\begin{table}[ht]
\begin{tabular}{c|c|c|c|c|c|c|c|c|c}
$m$ & $1$ & $2$ & $3$ & $4$ & $5$ & $6$ & $7$&$8$& $9$ \\\hline
$p^{(3)}$ & $2$ & $2.51$ & $2.83$ & $3$ &$3$ &$2.56$ &&& \\\hline
$p^{(4)}$ & $1.75$ & $2.11$ & $2.33$ & $2.46$ & $2.50$ &  $2.45$ & $2.23$&& \\\hline
$p^{(5)}$ & $1.63$& $1.93$ & $2.12$ & $2.24$ &$2.31$ &$2.33$ & $2.29$ & $2.17$ &$1.87$\\\hline
\end{tabular}
\vspace{5pt}
\caption{Approximate values of the largest root $p^{(k)}$ of the polynomial $h_m^{(k)}$ in \eqref{eq: h_m k} for $k\in[3,4,5]$. The blank entries correspond to values of $m$ and $k$ such that $h_m^{(k)}$ has no positive root.}\label{tab: values p0 k}
\end{table}
\endgroup

\noindent Now assume $k\ge 6$. We need the following claim.

\noindent{\bf Claim 1:} For any $k\ge 6$, there exists an integer $m^*\ge 1$ such that for all $m\ge m^*$ we have
\begin{equation}\label{eq: inequality u v}
u(k,m)^2\le v(k,m)e_{k-2}^{(k)}\,.
\end{equation}

Expanding the polynomials $u(k,m)$ and $v(k,m)$ and dividing by $k!$, the inequality \eqref{eq: inequality u v} is equivalent to
\begin{equation}\label{eq:quadratic m}
k!(k+m)^2-2(k+m)e_{k-1}^{(k)}-\frac{1}{2}m(m-2k+3)e_{k-2}^{(k)}+k!\left(\sum_{i=1}^k\frac{1}{i}\right)^2\le 0\,.
\end{equation}
The left-hand side of \eqref{eq:quadratic m} is the quadratic polynomial $H_k(m)=\alpha_2m^2+\alpha_1m+\alpha_0$, where
\begin{align}
\begin{split}
\alpha_2 &= k!-\frac{1}{2}e_{k-2}^{(k)}\\
\alpha_1 &= 2k!k+\frac{2k-3}{2}e_{k-2}^{(k)}-2e_{k-1}^{(k)}\\
\alpha_0 &= k^2k!-2k\,e_{k-1}^{(k)}+k!\left(\sum_{i=1}^k\frac{1}{i}\right)^2\,.
\end{split}
\end{align}

The polynomial $H_k(m)$ has at most two roots, and an upper bound $m^*$ for these roots can be obtained using the bounds provided by Lagrange or Cauchy \cite{hirst1997bounding}.
Since the leading coefficient $\alpha_2$ is negative for all $k\ge 6$, we conclude that $H_k(m)\le 0$ for all $m\ge m^*$. This concludes the proof of Claim 1.

Now using Proposition \ref{prop: property *} and Claim 1, we show that the polynomial $h^{(k)}_m(p)$ has no positive root for sufficiently large $m$, thus verifying (3).
To simplify furthermore, let
\[
s(p)\coloneqq\sum_{j=0}^{k-2} e_j^{(k)}p^{k-j}\,.
\]
Consider the product depending on the parameter $b\in\R$
\[
h^{(k)}_m(p)(p+b)=s(p)\,p +\left(u(k,m)+b\,\frac{s(p)}{p^2}\right )p^2+\left[v(k,m)+b\,u(k,m)\right]p+b\,v(k,m)\,.
\]
Note that all coefficients of $s(p)p$ are always positive.
If we choose $b\ge 0$ and $m\ge 2k-2$, then the constant term of $h^{(k)}_m(p)(p+b)$ is nonnegative.
Moreover, the inequality \eqref{eq: inequality u v} of Claim 1 and the fact that $u(k,m)<0$ for all $m\ge 1$ imply that there exists an integer $m^*\ge 1$ and a real number $b_0\ge 0$ such that 
\[
0\le-\frac{u(k,m)}{e_{k-2}^{(k)}}\le b_0\le -\frac{v(k,m)}{u(k,m)}\quad\forall\,m\ge m^*\,.
\]
This choice of $b_0$ implies that $v(k,m)+b_0\,u(k,m)\ge 0$ and $u(k,m)+b_0\,e^{(k)}_{k-2}\ge 0$ for all $m\ge m^*$. Therefore, all coefficients of $h^{(k)}_m(p)\,(p+b_0)$ are nonnegative for all $m\ge m^*$. Proposition \ref{prop: property *} with $g=p+b_0$ implies that $h^{(k)}_m(p)$ does not have any positive root for all $m\ge m^*$. Since $h^{(k)}_m(p)$ has either zero or two positive roots for $m\ge 2k-2$, then $h_m^{(k)}$ has exactly two positive roots for all $2k-2\le m\le m^*-1$. This completes the proof.\qedhere
\end{proof}

The following Macaulay2 code \cite{GS} computes the dimension of the $k$th-order factor analysis model. After defining the input variables and the components of the map $\varphi_k$ in \eqref{eq: def phi k}, we plug in random values of the parameters in the Jacobian matrix of $\varphi_k$, and we compute its rank:

{\small
\begin{Verbatim}[commandchars=\\\{\}]
dimFactorAnalysisModel = (k,p,m,F) -> (
    R := F[ε_(2,1)..ε_(k,p),
           δ_(2,1)..δ_(k,m),
           (\textcolor{magenta}{flatten} \textcolor{blue}{for} i \textcolor{blue}{in} 1..p \textcolor{blue}{list} \textcolor{blue}{for} j \textcolor{blue}{in} 1..\textcolor{magenta}{min}(i,m) \textcolor{blue}{list} λ_(i,j))];
    \textcolor{orange}{-- we construct the matrix Λ:}
    varδ := \textcolor{magenta}{flatten} \textcolor{blue}{for} i \textcolor{blue}{in} 2..k \textcolor{blue}{list for} j \textcolor{blue}{in} 1..m \textcolor{blue}{list} δ_(i,j);
    varΛ := \textcolor{magenta}{flatten} \textcolor{blue}{for} i \textcolor{blue}{in} 1..p \textcolor{blue}{list for} j \textcolor{blue}{in} 1..\textcolor{magenta}{min}(i,m) \textcolor{blue}{list} λ_(i,j);
    Λ0 := \textcolor{magenta}{mutableMatrix}(R,p,m);
    \textcolor{blue}{for} i \textcolor{blue}{in} 1..p \textcolor{blue}{do} \textcolor{blue}{for} j \textcolor{blue}{in} 1..min(i,m) \textcolor{blue}{do} Λ0_(i-1,j-1) = λ_(i,j);
    Λ := \textcolor{magenta}{matrix} Λ0;
    \textcolor{orange}{-- we define the map φ_k}
    ind := \textcolor{magenta}{apply}(k-1, s ->
    \quad \textcolor{magenta}{sort toList}(\textcolor{magenta}{set}(\textcolor{magenta}{subsets}(\textcolor{magenta}{flatten apply}(p, i -> \textcolor{magenta}{toList}(k:(i+1))),s+2))));
    \textcolor{blue}{for} r \textcolor{blue}{in} 2..k \textcolor{blue}{do for} i \textcolor{blue}{in} ind#(r-2) \textcolor{blue}{do} (
    \quad \textcolor{blue}{if} #set(i)==1 \textcolor{blue}{then} E_(\textcolor{magenta}{toSequence}({r}|i)) = ε_(r,i#0) \textcolor{blue}{else} E_(\textcolor{magenta}{toSequence}({r}|i)) = 0);
    φ := \textcolor{magenta}{apply}(k-1, s- > \textcolor{magenta}{apply}(ind#s, i ->
    \quad E_(\textcolor{magenta}{toSequence}({s+2}|i))+\textcolor{magenta}{sum}(m, l -> δ_(s+2,l+1)*\textcolor{magenta}{product}(s+2, j -> Λ_(i#j-1,l)))));
    Jφ := \textcolor{magenta}{diff}(\textcolor{magenta}{vars} R, \textcolor{magenta}{transpose matrix}\{\textcolor{magenta}{flatten} φ\});
    \textcolor{orange}{-- we plug in random values of the parameters:}
    subvarδ := \textcolor{magenta}{apply}(varδ, s -> s => \textcolor{magenta}{random}(F));
    subvarΛ := \textcolor{magenta}{apply}(varΛ, s -> s => \textcolor{magenta}{random}(F));
    subJφ := \textcolor{magenta}{sub}(\textcolor{magenta}{sub}(Jφ, subvarδ|subvarΛ),F);
    \textcolor{blue}{return} \textcolor{magenta}{rank} subJφ;
    )
\end{Verbatim}
}

The following lines check that the dimension computed coincides with the formula \eqref{eq: dim k} for the chosen input $(k,p,m,\F)$, where $\F$ denotes the base field (in our example, $\F=\Q$):

{\small
\begin{Verbatim}[commandchars=\\\{\}]
(k,p,m,F) = (3,4,3,\textcolor{purple}{QQ}); \textcolor{orange}{-- use your favourite input parameters}
d1 = dimFactorAnalysisModel(k,p,m,F);
d2 = (k-1)*p+(k-2)*m+\textcolor{magenta}{min}(p*m-\textcolor{magenta}{binomial}(m,2),\textcolor{magenta}{binomial}(p+k-1,k)-p);
d1 == d2 \textcolor{orange}{-- d1 = d2 = 20 for (k,p,m) = (3,4,3)}
\end{Verbatim}
}

\section*{Acknowledgements}

We would like to thank the unknown referees for their valuable comments, which revealed a mistake in the first version of Theorem \ref{thm: dimension k order factor analysis model}.\\
We thank the organizers of the joint event among Aalto University, Imperial College, and TU M\"unchen that took place in M\"unchen in March 2022, in particular Mathias Drton, Alexandros Grosdos, and Nils Sturma who presented open problems on factor analysis model, and for the interesting discussions that improved our paper.
The project was funded by the TUM Global Incentive Fund ``Algebraic Methods in Data Science''.\\
The second author conducted most of the research on this project at Aalto University, and both authors were partially supported by the Academy of Finland Grant No. 323416. The second author is currently supported by a KTH grant by the Verg foundation and Brummer \& Partners MathDataLab.

\bibliographystyle{alpha}
\bibliography{references.bib}
\end{document}